\documentclass{amsart}

\usepackage{amsmath}
\usepackage{graphicx}
\usepackage{hyperref}

\newtheorem{lemma}{Lemma}
\newtheorem{lem}{Lemma}
\newtheorem{remark}{Remark}
\newtheorem{proposition}{Proposition}
\newtheorem{theorem}{Theorem}
\newtheorem{corollary}{Corollary}

\newcommand{\deq}{\stackrel{d}{=}}
\newcommand{\mean}{\mathbb E}
\newcommand{\prb}{\mathbb P}
\newcommand{\hs}{\hat{s}}
\newcommand{\hatt}{\hat{t}}

\begin{document}
\title{Queues and risk models with simultaneous arrivals}

\author{E.S. Badila}
\address{Department of Mathematics and Computer Science, Eindhoven University of Technology, P.O. Box 513, 5600 MB Eindhoven, The Netherlands.} 
\email{e.s.badila@tue.nl}
\author{O.J. Boxma}
\email{boxma@win.tue.nl}
\author{J.A.C. Resing}
\email{resing@win.tue.nl}
\author{E.M.M.\ Winands}
\address{Korteweg de Vries Instituut voor Wiskunde, University of Amsterdam, P.O. Box 94248, 1090 GE Amsterdam, The Netherlands.}
\email{E.M.M.Winands@uva.nl}
\thanks{The authors gratefully acknowledge stimulating discussions with Yifan Xu (Binghamton University).
Serban Badila is being supported by Project 613.001.017 of the Netherlands Organisation for Scientific Research (NWO)}


\begin{abstract}
We focus on a particular connection between queueing and risk models in a multi-dimensional setting.
We first consider the joint workload process in a queueing model with parallel queues and simultaneous arrivals at the queues. For the case that the service times are ordered (from largest in the first queue to smallest in the last queue) we obtain the Laplace-Stieltjes transform of the joint stationary workload distribution.
Using a multivariate duality argument between queueing and risk models, this also gives the Laplace transform of the survival probability of all books in a multivariate risk model with simultaneous claim arrivals and the same ordering between claim sizes.

Other features of the paper include a stochastic decomposition result for the workload vector, and an outline how the two-dimensional risk model with a general two-dimensional claim size distribution (hence without ordering of claim sizes) is related to a known Riemann boundary value problem.
\end{abstract}


\maketitle

\begin{section}{Introduction}
There are several connections between queueing and risk models. A classical
result is that the ruin probability in the Cram\'{e}r-Lundberg risk model,
in which the arrival process of claims is a compound Poisson process,
is related to the workload (or waiting time) in an $M/G/1$ queue with the
same compound Poisson input. More precisely, denoting by $(R_t)_{t\geq 0}$
the surplus process in the Cram\'{e}r-Lundberg risk model, by $\tau$ the
time of ruin of this process and by $(V_t)_{t\geq 0}$ the workload process
in the corresponding $M/G/1$ queue, one has $\prb(\tau \leq t|R_0=u)$
$= \prb(V_t > u|V_0=0)$; in particular, the probability of ruin ever occurring
when starting at $u$ equals the probability that the steady-state workload
exceeds $u$. See, e.g., the nice geometric duality argument on page $46$
of Asmussen and Albrecher \cite{AsmussenRuin},
or Rolski et al. \cite{RSST}.

However, also other ruin-related performance measures have a counterpart
in queueing theory. By interpreting the interarrival times of the claims
as service times of the corresponding queue and the claim sizes as
interarrival times of the queue, the standard Cram\'{e}r-Lundberg model is
translated into a $G/M/1$ queue. The time to ruin in the
Cram\'{e}r-Lundberg model is now related to a busy period of the
corresponding queue, the deficit at ruin to an idle period and the
surplus just before ruin to the sojourn time of the last customer in a
busy period (see Frostig \cite{Frostig} and L\"{o}pker and Perry
\cite{Perry}).

In this paper our focus is on a connection between queueing and risk models
in a multi-dimensional setting. In particular, we look at the joint workload
process in a queueing model with parallel queues and simultaneous arrivals
at the queues. Under the condition that, with probability 1, the service
times of the customers arriving at the same time at the different queues
are ordered (i.e., the customer in queue 1 has the largest service time,
the customer in queue 2 the second largest service time, and so on)
we are able to find the Laplace-Stieltjes transform of the joint stationary
workload distribution in the different queues. Using a multivariate duality argument
between queueing and risk models, this immediately gives the
Laplace transform of the survival (non-ruin) probabilities in a multivariate risk
model with simultaneous claim arrivals (and the same ordering property for
the claim sizes of the simultaneous claims at the different books in the
model).

Queueing models with parallel queues and simultaneous arrivals are also
often called fork-join queues. These models have many applications in
computer-, communication- and production systems in which jobs are split
among a number of different processors, communication channels or machines.
Clearly, the queues in these models are dependent due to the simultaneous
arrivals. In general this makes an exact analysis of the model very hard.
Only in the case of two queues, exact results are available (see, e.g.,
Flatto and Hahn \cite{flattohahn}, Wright \cite{Wright}, Baccelli \cite{Baccelli},
De Klein \cite{DeKlein} and Cohen \cite{Cohen92}). We will come back to some of these exact results
in Section \ref{Section boundary value} of the paper. For the model with more than two servers no
exact analytical results are available in the literature. In this case,
bounds and approximations for several performance measures have been
developed, see e.g. \cite{BaccelliMS, RNANT88, NelsonT}.

Multivariate risk models with simultaneous claim arrivals have several applications in the area of ruin theory.
One example is provided by reinsurance models in which, whenever a claim arrives,
several insurance companies pay a part of the claim. Another example would
be a large insurance company with multiple lines of business, where
correlated claims arrive at the various business lines. Albeit in a
different area of risk management, analysis of the dependence between
the stochastic asset processes of several counter parties is also one of
the most challenging aspects in the field of credit risk. Especially, in
a two-dimensional setting one has to study the joint asset process of an
obligor and a guarantor in credit default swaps.

Avram, Palmowski and Pistorius \cite{APP2, APP1} have studied the joint ruin problem for the
special case of two insurance companies that divide between them both
claims and premia in some specific proportions. In particular, they derive
the double Laplace transform with respect to the two initial reserves
of the survival probabilities of the two companies.
Proportional claims are a special case of our ordered claims, and
we show in Section \ref{Sect rel_other models}
that their survival result indeed is a special case of our Formula (\ref{psi_final}).
One of the key observations in \cite{APP2, APP1} is that, due to the fact that
the companies divide the claims in some specific proportions, the
two-dimensional ruin problem may be viewed as a one-dimensional crossing
problem over a piecewise linear barrier. Badescu, Cheung and Rabehasaina
\cite{Badescu} have extended the two-dimensional model of \cite{APP2, APP1}
by allowing, next to the arrivals of claims for which the two insurers
divide the claim in some specific proportions, also extra arrivals of
claims which are fully paid by one of the insurers (e.g., insurer 1).
They show that under some conditions also in this model the previously
mentioned reduction to a 1-dimensional problem still holds. However, in
\cite{Badescu} the authors do not consider the double Laplace transform
with respect to the two initial reserves of the survival probabilities of
the two companies (their main focus is on the Laplace transform of the
time until ruin of at least one insurer).


The remainder of the paper is organized as follows: In Section \ref{Section duality} we present our model in detail and we provide the multivariate
duality argument. This duality argument allows a translation between results for the queueing model and results for the multivariate risk model.
Section \ref{SectionK_2} is dedicated to
the analysis of the 2-dimensional queueing model with ordered service times. After introducing the assumptions,
we derive the Laplace-Stieltjes transform of the joint
stationary workloads in the two queues and present a decomposition
theorem for the stationary workload in the two queues. In Section \ref{Section K} we extend the results of Section \ref{SectionK_2} to the $K$-dimensional queueing model.
Section \ref{Sect rel_other models} is dedicated to relations to other models. We present connections with tandem and priority queues, but also with a reinsurance problem
with proportional claim sizes. 
In Section \ref{Section boundary value} we discuss the case of a general two-dimensional service time (or claim size) distribution.
We indicate that the two-dimensional workload problem has been solved in the queueing literature.
The solution is very complicated; our ordered service times case is a degenerate case, but a case which has
the advantage of a much more explicit solution which offers more probabilistic insight -- and a case that can be generalized
to higher dimensions.
 Finally, Section \ref{section comments} outlines possible further research directions.

Among the main contributions of our paper, we mention an explicit result for the transform
of the joint workload (respectively, of the joint survival probability) and its extension to the $K$-dimensional model.
In addition, we mention the workload decomposition result. It seems to be new
in this setting, although similar results -- under the assumption of independent inputs -- were obtained for parallel queues (cf. \cite{Kella}).
From a more abstract perspective, another contribution of our paper is that it strengthens the links between
queueing and risk models, pointing out that certain results and methods in the literature (and in the present paper) for queues with simultaneous arrivals
are of immediate use in the risk setting, and vice versa.
\end{section}

\begin{section}{Multivariate Duality\label{Section duality}}

We consider a $K$-dimensional risk process in which claims arrive simultaneously in the K branches, according to a Poisson process with rate $\lambda$.
The claim sizes in the $K$ books are independent, identically distributed random vectors $(B_n^{(1)},...,B_n^{(K)})$, $n\geq 1$. In the sequel we denote with $(B^{(1)},...,B^{(K)})$
a random vector with the same distribution as $(B_1^{(1)},...,B_1^{(K)})$.

For the n$^{th}$ arriving claim vector, denote by $A_n$ the time elapsed since the arrival of the previous claim vector, so that the $A_n$ are independent and have an identical exponential
distribution with parameter $\lambda$.

Let $R_t^{(i)}$, $i=1,...,K$ be $K$ risk reserve processes with initial capitals $u_i$, premium rates $c^{(i)}$ and the \textit{same} arrival instants $\sigma_n$, $n\geq 1$. We have  $A_n=\sigma_n-\sigma_{n-1}$ and $\sigma_0=0$ (no delay). Then

\begin{equation}
R^{(i)}_t=u_i+\sum_{j=1}^{n(t)} (c^{(i)}A_j-B^{(i)}_j)+c^{(i)}\left(t-\sigma_{n(t)}\right),
\label{definerisk}
\end{equation}
where $n(t)$ is the number of arrivals before $t$.
Let $\tau^{(i)}(u_i)=\inf\left\{t>0:R^{(i)}_t<0\right\}$ be the times to ruin.

In connection with the ruin process, we consider $K$ parallel $M/G/1$ queues with simultaneous (coupled) arrivals and correlated service requirements. As in the ruin setting, $A_n$ are the interarrival times of customers in the $K$ queues and the vector $(B^{(1)},...,B^{(K)})$ denotes the generic service requirements.
 The speed of server $i$ is denoted by $c^{(i)}$, meaning
that server $i$ handles $c^{(i)}$ units of work per time unit, $i=1,\dots,K$.

Furthermore we denote by $\rho_i:=\lambda \mean(B^{(i)})$ the load of queue $i$, $i=1,...,K$ and we assume that $\rho_i<c^{(i)}$, to ensure that all queues can handle the offered traffic.
These conditions imply positive safety loading in the ruin setting.


From the queueing perspective, let $(V_t^{(1)},...,V_t^{(K)})$ be the workload vector at time $t$ in the system or, if we consider the $n^{th}$ arrival epoch, this is the workload $(V_n^{(1)},\dots,V_n^{(K)})$ seen by the customers of the $n^{th}$ batch arrival.
Remark that $V_n^{(i)} = c^{(i)} W_n^{(i)}$, with $W_n^{(i)}$ the waiting time of the $n^{th}$
arrival in queue $i$.
Under the stability conditions above, 
the vectors $(V_t^{(1)},...,V_t^{(K)})$ and $(V_n^{(1)},...,V_n^{(K)})$
converge in distribution to the steady-state joint workload 
at arbitrary  epochs and at arrival epochs, respectively. Due to the PASTA property these vectors are equal.
Similarly, the vector $(W_n^{(1)},\dots,W_n^{(K)})$ converges in distribution to the steady state waiting time.
We denote the Laplace-Stieltjes transform (LST) of the steady-state workload vector:
\[\psi(s_1,s_2,...,s_K) := \mean(e^{-s_1V^{(1)}-s_2V^{(2)}-...-s_KV^{(K)}}).\]


For the multidimensional ruin process defined in (\ref{definerisk}), consider a dual workload process with $V^{(i)}_N$, $i=1,...,K$  the workload seen upon arrival by the $N^{th}$ customer in $K$ initially empty queues with the time reverted arrival process (the arrival epochs are the same for all the systems):

\[\sigma^*_n=\sigma_{N-n+1},\;\;\; (A^*_n=A_{N-n+1}), \;\; n=1,...,N;\]
service time of customer $n$ at queue $i$: $B^{*(i)}_n=B^{(i)}_{N-n+1}$, $n=1,...,N$
(time reverted service time) (cf. \cite{AsmussenRuin}).


\vspace{0.2cm}
The following lemma shows that the well-known duality result (cf. \cite{AsmussenRuin}, p. 46) between the Cram\'er-Lundberg model and the $M/G/1$ queue can be extended to the multivariate risk model
and the queueing model with simultaneous arrivals. 
Here the connection is between the various possibilities to be ruined (i.e we may have ruin in all books or precisely in one, at least in one, etc.)
The results below are presented for the case $K=2$, but can be directly extended to the general case.
\begin{lemma}\label{Duality Lemma}
 The following identities  hold:
\begin{itemize}
\item[(a)] $\left\{V^{(1)}_N >u_1 \wedge V^{(2)}_N >u_2 \right\} = \left\{\tau^{(1)}(u_1)\leq\sigma_N \wedge \tau^{(2)}(u_2)\leq\sigma_N\right\}$

\item[(b)] $\left\{V^{(1)}_N \leq u_1 \wedge V^{(2)}_N \leq u_2 \right\} = \left\{\tau^{(1)}(u_1)>\sigma_N \wedge \tau^{(2)}(u_2)>\sigma_N \right\}$

\item[(c)] $\left\{V^{(1)}_N >u_1 \wedge V^{(2)}_N \leq u_2 \right\} = \left\{\tau^{(1)}(u_1)\leq\sigma_N \wedge \tau^{(2)}(u_2)>\sigma_N \right\}$

\item[(d)] $\left\{V^{(1)}_N\leq u_1 \wedge V^{(2)}_N > u_2 \right\} = \left\{\tau^{(1)}(u_1)>\sigma_N \wedge \tau^{(2)}(u_2)\leq \sigma_N  \right\}$

\end{itemize}
The above relations are path-wise identities.
\end{lemma}

\begin{proof}
The following identities hold for the cylinder sets: \[\{V_N^{(i)}>u_i\}=\{ \tau^{(i)}(u_i)\leq \sigma_N\}. \]
This follows directly from Asmussen and Albrecher (\cite{AsmussenRuin}, page 46) for the 1-dimensional problem, and is a special case of the duality in Siegmund \cite{Siegmund_duality}.





If we intersect the above identities, we obtain $(a)$. $(b)$ follows by intersecting their complements, and $(c)$ and $(d)$ by subtracting $(b)$ and $(a)$ respectively, from the complements of the above cylinder sets.
This concludes the proof.\end{proof}


If we let $N\rightarrow \infty$ in $(b)$ of Lemma \ref{Duality Lemma}, we obtain the infinite horizon joint survival probability

\begin{equation} \lim_{N\rightarrow\infty} \mathbb P(V^{(1)}_N\leq u_1 \wedge V^{(2)}_N\leq u_2)=\mathbb P(\tau^{(1)}(u_1)=\infty\wedge \tau^{(2)}(u_2)=\infty). \label{infinite horizon duality} \end{equation}
Denote the righthand side by $\xi(u_1,u_2)$. This is the joint survival function, for initial capital $(u_1,u_2)$.
By PASTA, we can replace the steady state workload at arrival epochs with the steady state workload at arbitrary epochs in $(\ref{infinite horizon duality})$.



\vspace{0.2cm}

Let \[\xi_*(s,t):=\int e^{-sx_1-tx_2}\xi(x_1,x_2)dx_1dx_2\]
be the Laplace transform (LT) of the joint survival function. Via $(\ref{infinite horizon duality})$, this is also the LT of the c.d.f. of the joint workload in steady state. By a simple integration by parts, we have the following relation with the LST of the workload:
\begin{equation}\xi_*(s,t)=\frac{1}{st}\psi(s,t).\label{LT-LST}\end{equation}
\end{section}

\begin{section}{The analysis of the two-dimensional problem} \label{SectionK_2}
In this section 
we derive the transform of the joint steady state workload process of the two-dimensional queueing model with simultaneous arrivals, as introduced in Section \ref{Section duality} .
We also present a probabilistic interpretation of the quantities involved in the formula of the joint workload. The results are of immediate relevance for the corresponding insurance problem, via
the duality outlined in the previous section.



Before we start with the analysis, we make the following simplifying assumption.
\\
{\bf Assumption 1.}
All premium rates, respectively all service speeds, are $1$, viz., $c^{(1)} = \dots = c^{(K)} =1$.
\\
The following observation shows that this assumption is not restrictive.
If we divide all terms in the righthand side of (\ref{definerisk}) by $c^{(i)}$, we arrive
at a new risk model with initial capital $u_i/c^{(i)}$ and  claim size $B^{(i)}/c^{(i)}$ and \textit{unit} premium rates.
Similarly, in the corresponding queueing model the service times at queue $i$ are also divided by $c^{(i)}$ and the service speeds are equal to 1.
This will not change the $n^{th}$ waiting time $W_n^{(i)}$ at queue $i$, but the
workload $V_n^{(i)}$ at the $n^{th}$ arrival epoch is divided by $c^{(i)}$.
Also the times to ruin are preserved, hence the identities in Lemma~\ref{Duality Lemma} from the previous section remain unchanged.

\noindent The LST of the joint service time/claim size vector is denoted by
\[\phi(s,t) := \mean(e^{-sB^{(1)}-tB^{(2)}}).\]
Our key assumption is the following:
\\
{\bf Assumption 2.}
$\prb(B^{(1)}\geq B^{(2)})=1$.
In view of the above discussion, in the case of speeds $c^{(i)}$ our assumption would be $\prb(B^{(1)}/c^{(1)}\geq B^{(2)}/c^{(2)})=1$.

\begin{remark}\label{rmk 2dedicated}
This model allows for a dedicated Poisson arrival stream into queue 1.
Merging this separate arrival process with the simultaneous arrival process at queue 1,
the distribution of $B^{(2)}$ will have an atom in 0, which is the probability that a
dedicated Poisson arrival happens instead of a simultaneous one (see Badescu et al. \cite{Badescu}
for a reinsurance model with both dedicated and simultaneous arrivals).
\end{remark}

We are interested in the joint stationary distribution of the amount
of work in the two queues 
\[
\psi(s,t) := \mean (e^{-sV^{(1)}-tV^{(2)}}).
\]


This can be obtained in the following way.
Consider the amount of work in queue $i$ just
before the arrival of customer $n$.
We have the following recursion for the random variables $(V^{(1)}_n,V^{(2)}_n), n=1,2,\ldots$
\begin{eqnarray*}
(V^{(1)}_{n+1},V^{(2)}_{n+1}) &=&
({\rm max}(V^{(1)}_n + B^{(1)}_n - A_n,0),{\rm max}(V^{(2)}_n + B^{(2)}_n - A_n,0) ).
\end{eqnarray*}

Or, for the LST
\[
\psi_n(s,t) = \mean \left(e^{-sV^{(1)}_n-tV^{(2)}_n}\right), \quad n=1,2,\ldots,
\]
this gives after straightforward calculations
\begin{eqnarray}
\label{psi_n}
\psi_{n+1}(s,t) &=& {\textstyle \frac{\lambda}{\lambda-s-t}}\left(\phi(s,t)\psi_n(s,t)-
\phi(s,\lambda-s)\psi_n(s,\lambda-s)\right) \nonumber \\
&+&
{\textstyle \frac{\lambda}{\lambda-s}}\left(\phi(s,\lambda-s)\psi_n(s,
\lambda-s)-
\phi(\lambda,0)\psi_n(\lambda,0)\right) \nonumber\\
&+&
\phi(\lambda,0)\psi_n(\lambda,0).
\end{eqnarray}

\noindent Under the stability condition $\rho_1<1$,
$\psi(s,t) := \lim_{n \to \infty} \psi_n(s,t)$ exists and
\begin{eqnarray}
\label{psi}
\left(1-{\textstyle \frac{\lambda \phi(s,t)}{\lambda-s-t}}\right)\psi(s,t) &=&
\left({\textstyle \frac{\lambda}{\lambda-s}-\frac{\lambda}{\lambda-s-t}}\right)
\phi(s,\lambda-s)\psi(s, \lambda-s) \nonumber \\
&+&
\left(1-{\textstyle \frac{\lambda}{\lambda-s}}\right)
\phi(\lambda,0)\psi(\lambda,0).
\end{eqnarray}

\noindent If we let $A$ denote a generic interarrival time, then due to the PASTA property,
\begin{equation}
\label{psi_lambda}
\phi(\lambda,0)\psi(\lambda,0)=\mathbb P(V^{(1)}+B^{(1)}\leq A)= \prb (V^{(1)}=0)= 1- \rho_1.
\end{equation}
This is the probability that queue 1 is empty at an arbitrary time instant.
\vspace{0.2cm}

On the regularity domains of $\psi(s,t)$ and $\phi(s,t)$: We remark that, because of the dependence $\mathbb P(B^{(1)}\geq B^{(2)})=1$, we can rewrite the transform of the joint service times as:

\[\phi(s,t)=\mathbb Ee^{-s(B^{(1)}-B^{(2)})-(s+t)B^{(2)}}=:\tilde\phi(s,s+t),\]

\noindent and this function is always regular in $\mathcal Re\;s>0$, $\mathcal Re(s+t)>0$. If we consider $(B^{(1)},B^{(2)})$ subject to $B^{(1)}\geq B^{(2)}$ a.s., $\phi(s,t)$ may not be regular beyond this domain. More precisely, if $B^{(2)}$ has a heavy-tailed distribution, this implies that $B^{(1)}$ is also heavy tailed because of the dependence structure. In this case $\phi(s,t)$ cannot be extended beyond $\mathcal Re\;s\geq 0$, $\mathcal Re\;(s+t)\geq 0$. Similar considerations hold for $\psi(s,t)$ because we must also have $\mathbb P(V^{(1)}\geq V^{(2)})=1$.

By Lemma \ref{Rouche} in the Appendix, $\forall s$ with $\mathcal Re$ $s>0$, there is a unique $t(s)$, well defined and analytic in $Re\;s>0$, such that
$\lambda \phi(s,t(s)) = \lambda -s-t(s)$. Hence $(s,t(s))$ is a zero of $\left(1-{\textstyle \frac{\lambda \phi(s,t)}{\lambda-s-t}}\right)$  in $(\ref{psi})$, which is in the regularity
domain of $\psi(s,t)$. Then the righthand side of (\ref{psi}) is also zero, i.e.
\begin{equation}\label{intermed_rel}
\lambda t(s) \phi(s,\lambda-s)\psi(s,\lambda-s) = -s(\lambda-t(s)-s)
\phi(\lambda,0)\psi(\lambda,0).
\end{equation}
If we substitute this in (\ref{psi}) and use (\ref{psi_lambda}), we
obtain

\begin{equation}
\label{psi_final}
\psi(s,t) = (1-\rho_1) \frac{s}{s+t-\lambda(1-\phi(s,t))}
\cdot\frac{t(s)-t}{t(s)}.
\end{equation}

\vspace{0.2cm}

\paragraph{\textbf{The interpretation of the Rouch\'e zero $t(s)$.}}

 Assume that a customer that starts a busy period $BP^{(2)}$ in queue 2 demands work $x$ in queue 2 and work $x+y$ in queue 1.
 During the service time of this customer in the second queue, there are Poisson$(\lambda x)$ arriving customers, each one  of these generating an
 i.i.d. busy sub-period with the same distribution as $BP^{(2)}$ in queue 2. So if we denote with $U$ the extra work in the first queue, at the end of a busy period in the second queue,
 and with $U^*(s)$ its Laplace-Stieltjes transform, we have the identity:

\begin{eqnarray}
 U^*(s) &=& \int_{x=0}^{\infty}\int_{y=0}^{\infty}e^{-sy}\sum_{k=0}^\infty \frac{(\lambda x)^k}{k!}e^{-\lambda x}[U^*(s)]^k\; d\mathbb P(B^{(1)}-B^{(2)}\leq y, B^{(2)}\leq x). \nonumber
\end{eqnarray}
 The powers of $U^*(s)$ correspond to the extra work contributions at the end of the busy sub-periods started during the service time of the first customer in the busy period $BP^{(2)}$.
 We can rewrite the above identity as:

 \begin{equation}\label{U_star}U^*(s)=\tilde\phi(s,\lambda [1-U^*(s)])=\phi(s,\lambda[1-U^*(s)]-s).\end{equation}
 Comparing this with the equation in Lemma \ref{Rouche} in terms of $\tilde\phi(s,s+t)$, we have:

\[ \left\{ \begin{array}{l}\lambda\tilde\phi(s,s+t(s))=\lambda-(s+t(s))\\
 \lambda\tilde\phi(s,\lambda[1-U^*(s)])=\lambda U^*(s). \end{array}\right.\]

 We may assume w.l.o.g. that $\prb(B^{(1)}>B^{(2)})>0$, otherwise the two queues are a.s. identical, which is not interesting. Then it follows that the real part of $\lambda(1-U^*(s))$
 is positive, and we must have $s+t(s)=\lambda(1-U^*(s))$
 because  the solution obtained in Lemma \ref{Rouche} is unique in the region $\mathcal Re$ $(s+t)>0$. We have thus proved:
\begin{proposition}\label{Prop_Rouche} The relation between $t(s)$ and the transform of the extra workload in queue 1 at the end of a busy period in the shortest queue is

\begin{equation}\lambda U^*(s)=\lambda-(s+t(s)).\label{rel}\end{equation}
The transform of the joint workload in the two systems becomes

\[\psi(s,t) = (1-\rho_1) \frac{s+t-\lambda(1-U^*(s))}{s+t-\lambda(1-\phi(s,t))}
\cdot\frac{s}{s-\lambda(1-U^*(s))}.\]

\label{t(s)}
\end{proposition}

\vspace{0.2cm}

\paragraph{\textbf{The workload decomposition.}}
Based on Proposition \ref{Prop_Rouche}, we show that the steady-state workload decomposes into an independent sum
of a modified workload and an additional term,
which represents the steady-state workload in a classical M/G/1 queue.

We start the joint workload process and let it run until the end of each busy period in the queue with the smallest workload. At this random time instant, we remove the extra content in queue 1, which has the largest workload of the two. Let us denote this modified joint workload process as $(\tilde V^{(1)}, V^{(2)})$. 
Then at the arrival instants of customers in the two queues, the recurrence relation holds:

\[(\tilde V_{n+1}^{(1)}, V_{n+1}^{(2)})=\left\{\begin{array}{ll}(\tilde V_{n}^{(1)}+B_n^{(1)}-A_n, V_{n}^{(2)}+B^{(2)}_n-A_n),& \mbox{ if } A_n< V_n^{(2)}+B_n^{(2)} \\
(0,0), & \mbox{ if } A_n\geq V_n^{(2)}+B_n^{(2)}.\end{array}\right.\]
Remark that marginally, the shortest queue evolves unchanged.


\noindent If we have ergodicity then in steady state, the above recurrence becomes:
\[ (\tilde V^{(1)},V^{(2)})\deq \left\{\begin{array}{ll}(\tilde V^{(1)}+B^{(1)}-A,V^{(2)}+B^{(2)}-A), & \mbox{ if } A< V^{(2)}+B^{(2)} \\
(0,0), &\mbox{ if } A\geq V^{(2)}+B^{(2)}.\end{array}\right.  \]

Here and in the following, $\deq$ denotes equality in distribution.
If we rewrite this in terms of LST's, we obtain the following functional equation for $\tilde\psi(s,t):=\mathbb Ee^{-s\tilde V^{(1)}-t V^{(2)}}$:

\[(1-\frac{\lambda\phi(s,t)}{\lambda-s-t})\tilde\psi(s,t)=(1-\rho_2)-\frac{\lambda}{\lambda-s-t}\tilde\psi(s,\lambda-s)\phi(s,\lambda-s),\]
where $1-\rho_2=\prb(V^{(2)}=0)$.

\vspace{0.2cm}

Now follows a similar analysis as for $\psi(s,t)$. We already know from the Rouch\'e problem that $t(s)$ from Lemma \ref{Rouche} is a zero of
$(1-\frac{\lambda\phi(s,t)}{\lambda-s-t})$. We also have $\tilde V^{(1)}\geq V^{(2)}$ a.s. (even if we take out the extra workload at the largest queue at the end of each busy period,
$\tilde V^{(1)}$ is still at least as large as $V^{(2)}$ in the long run), therefore $(s,t(s))$ is in the regularity domain of $\tilde\psi(s,t)$ and therefore,
at the point $(s,t(s))$, the right-hand side of the above identity is equal to zero:

\[ \tilde\psi(s,\lambda-s)\phi(s,\lambda-s)=(1-\rho_2)\frac{\lambda-s-t(s)}{\lambda}.  \]
Substituting back in the original identity, yields:

\begin{equation}\label{psi_tilde}
\tilde\psi(s,t)=(1-\rho_2)\frac{s+t-\lambda(1-\phi(s,t(s)))}{s+t-\lambda(1-\phi(s,t))}.
\end{equation}

This is a 2-dimensional Pollaczek-Khinchine type of representation. From an analytic point of view, the role of the numerator is to cancel the unique pole of the denominator in the
region $\mathcal R$e $(s+t)>0$.

Substitute $t(s)$ from Proposition \ref{t(s)} and $\tilde\psi$ from (\ref{psi_tilde}) into (\ref{psi_final}):

\begin{equation} \psi(s,t)=\frac{1-\rho_1}{1-\rho_2} \frac{s}{s-\lambda [1-U^*(s)]}  \tilde\psi(s,t).  \label{decomp}\end{equation}

We can now state the main result:
\begin{theorem}[Work decomposition]
\label{work decomp}
In steady state, we have the following representation of the joint workload at the two queues as an independent sum:
\[(V^{(1)},V^{(2)})\deq(\tilde V^{(1)},V^{(2)})+(V^{(1),1},0),\] where $V^{(1),1}$ is the workload in an independent,
virtual M/G/1 queue with arrival rate $\lambda$ and service requirements distributed as $U$, the extra workload at the end of a busy period $BP^{(2)}$ in the shortest queue.

\end{theorem}

\begin{proof}It suffices to remark that the factor
\[\frac{1-\rho_1}{1-\rho_2} \frac{s}{s-\lambda [1-U^*(s)]}=\mathbb E e^{-sV^{(1),1}}\] in (\ref{decomp}) is the Pollaczek-Khinchine formula for the
transform of the workload in the virtual M/G/1 queue with service time distribution $U$.
This virtual queue is obtained by contracting the busy periods in the initial shortest queue, so that an arrival in the virtual queue happens at the end of this busy period and the interarrival
time is then the idle period in the initial queue, and so is exponentially distributed.

To see that indeed $\frac{1-\rho_1}{1-\rho_2}$ is the atom of $V^{(1),1}$ at 0, differentiate the identity for $U^*(s)$ in (\ref{U_star}):
\[ \mathbb E(U)=-\frac{d}{ds}\phi(s,\lambda(1-U^*(s)-s))_{|s=0}=\mathbb E (B^{(1)}-B^{(2)})+\lambda \mathbb E B^{(1)}\mathbb E(U) \] so that $1-\lambda\mathbb E(U)=\frac{1-\rho_1}{1-\rho_2}$.\end{proof}

\end{section}

\begin{section}{Relation with other models}\label{Sect rel_other models}

In this section we point out how the results of the previous section are related to results for a risk model with proportional reinsurance, a particular
tandem fluid model and with a particular priority queue. We start by showing that (\ref{psi_final})  generalizes a result obtained in \cite{APP1}, for the risk setting.

\paragraph{\textbf{The case of proportional reinsurance.}}
In \cite{APP1} the joint reserve process $(R^{(1)},R^{(2)})$ is of the form: $R^{(i)}(t)=u_i+ c^{(i)} t/\delta_i - S(t)$. Here $S(t)$ is a common Compound Poisson input process with generic claim sizes $\sigma$ and $c^{(i)}$ are the premium rates. The claims are being divided in fixed proportions $\delta_i$, respectively.

To bring this closer to our setting in Section \ref{SectionK_2}, normalize the income rates: i.e. we consider $(\frac{1}{p_1} R^{(1)},\frac{1}{p_2}R^{(2)})$ with $p_i=\frac{c^{(i)}}{\delta_i}$.
The assumption in \cite{APP1} is that $p_1>p_2$, which means that, in our notation, the claim sizes are $B^{(1)}:=\frac{1}{p_1}\sigma<\frac{1}{p_2}\sigma=:B^{(2)}$.
Remark that the inequality between the $B^{(i)}$'s is reversed here (which means the role of the arguments in our transforms is interchanged, especially the Rouch\'e zero).

Let us recall the main formula in \cite{APP1} (Formula (23)):
\begin{equation}
\psi_{*R^{(1)},R^{(2)}}(p,q)= \frac{\kappa_2(0+)'}{p(\kappa_1(p+q) -q(p_1-p_2))} \frac{q+p-q^+(q(p_1-p_2))}{q-q^+(q(p_1-p_2))}. \label{MP}
\end{equation}
The relation between the ruin times of $(R^{(1)},R^{(2)})$ and $(\frac{1}{p_1} R^{(1)},\frac{1}{p_2}R^{(2)})$ is
\[\tau_{\frac{1}{p_1} R^{(1)},\frac{1}{p_2}R^{(2)}}(u_1,u_2)= \tau_{R^{(1)},R^{(2)}}(p_1u_1,p_2u_2).\]
Hence the relation to the LT coordinates used in (\ref{LT-LST}) is $s=p_1p$, $t=p_2q$.
 From this, the relation between the LT of the survival functions becomes after a change of variables:
 \begin{equation}\psi_{*\frac{1}{p_1}R^{(1)},\frac{1}{p_2}R^{(2)}}(s,t)=\frac{1}{p_1p_2}\psi_{*R^{(1)},R^{(2)}}(p,q).\label{LTrel}\end{equation}

\vspace{0.2cm}
\begin{itemize}
 \item{} $\kappa_i(\alpha)$ is the Laplace exponent of the Compound Poisson process with drift $p_i$ per unit time. This means
 \[ \kappa_i(\alpha)= p_i\alpha-\lambda(1-\mean e^{-\alpha\sigma}). \]

Because of the linear dependence between the $B^{(i)}$'s, their LST has the form
$\mean e^{-sB^{(1)}-tB^{(2)}}= \phi(s,t)=:\phi_{B^{(1)}}(s+\frac{p_1}{p_2}t)$.
 \item{} $q^+(q)$ is the largest root of the equation $\kappa_1(\alpha)=q$.
Then $q^+(q(p_1-p_2))$ solves:
\[ p_1\alpha - \lambda(1-\mean e^{-\alpha p_1B^{(1)}})=q(p_1-p_2). \]
Remark that if we set $\alpha=p+q$, the above becomes:
 \[p_1p+p_2q -\lambda(1 - \phi_{B^{(1)}}(p_1p+p_1q) ) =0,  \] or, written in the $(s,t)$-coordinates, this becomes the equation in Lemma \ref{Rouche}
 (with $s$ and $t$ interchanged). Hence the relation between the zeroes in the two notations is: $s(t)=p_1(\alpha-q)=p_1[q^+(q(p_1-p_2))-q]$.

\noindent The constant $\kappa(0+)'=p_2-\lambda\mean B^{(2)}= p_2(1-\rho_2)$ is the probability that the queueing system is empty in steady state (now the second queue has a higher workload).
\end{itemize}

 In conclusion, $(\ref{MP})$ written via (\ref{LTrel}) and (\ref{LT-LST}) in the $(s,t)$  coordinates becomes Formula (\ref{psi_final}):

 \[\psi(t,s) =  \frac{s(1-\rho_2)}{s+t-\lambda(1-\phi_{B^{(1)}}(s+\frac{p_1}{p_2}t))}\cdot\frac{s-s(t)}{-s(t)},\]
with the arguments $s$ and $t$ interchanged.

\paragraph{\textbf{Relation with work on tandem fluid queues.}}
We now show that the workload model with ordered service times is equivalent with
a particular tandem fluid queue. That is a model of two queues in series, in which the
outflow from the first queue is a fluid, i.e., there is continuous outflow when the server is working
(instead of customers leaving one by one). Such tandem fluid queues have been studied by various authors, see in particular
\cite{Kella}.
Consider the following two-station tandem fluid network with {\it independent}
compound Poisson input at the two stations (with arrival rate $\lambda_i$
and Laplace-Stieltjes transform of the service times $B_i^*(\cdot), i=1,2$).
Then Theorem 4.1 of Kella \cite{Kella} gives the Laplace-Stieltjes
transform of the steady-state fluid levels
$W_1$ and $W_2$ in the two nodes:
\begin{equation}
\psi_W(\alpha_1,\alpha_2) = \mean\left(e^{-\alpha_1 W_1 -\alpha_2 W_2}\right) =
\frac{(1-\rho_1-\rho_2)\alpha_2}{\phi_1(\alpha_1)-\phi_1(\hat{\eta}_2(\alpha_2))} \cdot \frac{\alpha_1-\hat{\eta}_2(\alpha_2)}{\alpha_2-\hat{\eta}_2(\alpha_2)},
\end{equation}
with
\begin{itemize}
\item $\rho_i = \lambda_i \mean(B_i)$,
\item $\phi_1(\alpha_1) = \alpha_1 - \eta_1(\alpha_1)$,
\item $\eta_i(\alpha_i) = \lambda_i(1-B_i^*(\alpha_i))$,
\item $\hat{\eta}_2(\alpha_2)$ the solution of $\phi_1(\hat{\eta}_2(\alpha_2))
= \eta_2(\alpha_2)$.
\end{itemize}
Alternatively, the last relation can also be formulated as: $\hat{\eta}_2(\alpha_2)$ is the solution of
\[
\lambda_1 B_1^*(\hat{\eta}_2(\alpha_2)) + \lambda_2 B_2^*(\alpha_2)=
\lambda_1 + \lambda_2 - \hat{\eta}_2(\alpha_2).
\]
This system is related to our model with arrival rate $\lambda= \lambda_1+ \lambda_2$
and Laplace-Stieltjes transform of service requirements
\[
\phi(s,t) = \frac{\lambda_1}{\lambda_1+\lambda_2} B_1^*(s+t)
          + \frac{\lambda_2}{\lambda_1+\lambda_2} B_2^*(s).
\]

\vspace{0.2cm}

\begin{figure}[!htpb]
\caption{Tandem fluid queue}
\label{tandem fluid pic}
\centering
\includegraphics[scale=0.5]{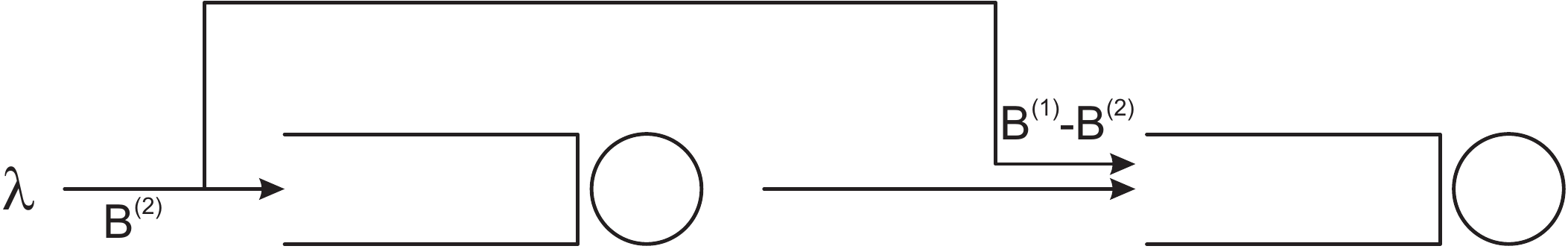}
\end{figure}

The corresponding notation is: $B_1\deq B^{(2)}$ and $B_2\deq B^{(1)}-B^{(2)}$.
Here $W_1$ in the tandem model corresponds to the workload in the smallest
queue in our model and $W_1+W_2$ in the tandem model corresponds to the
workload in the largest queue in our model. So we have
\begin{eqnarray*}
\psi(s,t) &=& \mean\left(e^{-sV_1-tV_2}\right) = \mean\left(e^{-s(W_1+W_2)-tW_1}\right) = \psi_W(s+t,s) \\
&=&
\frac{(1-\rho_1-\rho_2)s}{s+t - \lambda_1(1-B_1^*(s+t))- \lambda_2(1-B_2^*(s))} \cdot \frac{s+t-\hat{\eta}_2(s)}{s-\hat{\eta}_2(s)}.
\end{eqnarray*}
Now remark that
\begin{itemize}
\item The total traffic offered to the largest queue is $\rho_1+\rho_2$,
      so indeed the factor $1-\rho_1-\rho_2$ in \cite{Kella} corresponds to the factor
      $1-\rho_1$ in (\ref{psi_final});
\item $\lambda(1-\phi(s,t)) = \lambda_1(1-B_1^*(s+t)) + \lambda_2(1-B_2^*(s))$;
\item $\lambda \phi(s,t(s)) = \lambda_1 B_1^*(s+t(s)) + \lambda_2 B_2^*(s)
                            = \lambda_1 + \lambda_2 - (s+t(s))$, so indeed
      $\hat{\eta}_2(s)$ corresponds to our $s+t(s)$.
\end{itemize}
We conclude that (\ref{psi_final}) coincides with Theorem 4.1 of
\cite{Kella} in the case of independent compound Poisson input. Kella's
result is more general in the sense that he has L\'{e}vy input instead of
compound Poisson input. Our result is more general in the sense that we
have {\it dependent} compound Poisson input.

\paragraph{\textbf{Relation with work on priority queues.}}
As was already noticed in Kella \cite{Kella}, but also in several other
places in the literature, the tandem fluid network described above is also related to a priority queue with
preemptive resume priorities. Hence the same holds for our workload model. Consider the following model with two types
of customers where customers of type-$i$ arrive according to a Poisson
process with rate $\lambda_i$ having service times with Laplace-Stieltjes
transform $B_i^*(\cdot), i=1,2$. Assume furthermore that customers of
type-$1$ have preemptive resume priority over customers of type-$2$.
If we denote by $Y_1$ and $Y_2$ the steady-state workloads in the two
queues, then $Y_1$ and $Y_2$ are related to $W_1$ and $W_2$ in the tandem
fluid network. The Laplace-Stieltjes transform of the
steady-state workloads in the two queues satisfies
\[
\psi_Y(s,t) = \mean\left(e^{-s Y_1-t Y_2}\right)
              = \mean\left(e^{-s W_1-tW_2}\right)
              = \mean\left(e^{-s V_2-t(V_1-V_2)}\right) = \psi_V(t,s-t)
\]
where again in our model we have to take
arrival rate $\lambda= \lambda_1+ \lambda_2$
and Laplace-Stieltjes transform of service requirements
\[
\phi(s,t) = \frac{\lambda_1}{\lambda_1+\lambda_2} B_1^*(s+t)
          + \frac{\lambda_2}{\lambda_1+\lambda_2} B_2^*(s).
\]
We conclude that (\ref{psi_final}) also gives the
Laplace-Stieltjes transform of a priority queue. Again our result is
more general in the sense that we have {\it dependent} compound
Poisson input (i.e., we can have arrivals of customers who have
both low and high priority work).

\end{section}

\begin{section}{The $K$-dimensional problem\label{Section K}}

\paragraph{} In this section we consider the $K$-queue system with simultaneous arrivals.
We give the transform for the steady-state joint workload and we show that the decomposition in Theorem \ref{work decomp}
extends to this case if we preserve the ordering between the service requirements/claim sizes. We use an iterative argument and for this purpose, the decomposition in Section \ref{SectionK_2}
 will be the starting point; the iteration step is essentially done with the help of Lemma \ref{Work conservation} below as a work conservation identity.

\paragraph{} We thus consider $K$ parallel $M/G/1$ queues, numbered $1$ to $K$, respectively, with simultaneous (coupled)
arrivals and correlated service requirements.
The arrival process is again a Poisson process with rate $\lambda$. The
service requirements of successive customers at the K queues  are independent, identically
distributed random vectors $(B_n^{(1)},\dots,B_n^{(K)}), n\geq 1$.
Let $(B^{(1)},\dots,B^{(K)})$ be a generic random vector with the same distribution
as $(B_1^{(1)},\dots, B_1^{(K)})$. The LST of the service time/claim size vector is denoted by
\[
\phi(s_1,\dots,s_K) := \mean(e^{-s_1B^{(1)}-\dots-s_KB^{(K)}}).
\]
\noindent The essential assumption in the model extends Assumption 2 for the $2$-dimensional problem: 
\[\prb(B^{(1)}\geq B^{(2)}\geq\dots\geq B^{(K)})=1.\]
Furthermore we denote by $\rho_i := \lambda \mean B^{(i)}$, $i=1,\dots,K$,  the load of queue $i$ and
we assume that $\rho_1 < 1$ (hence $\rho_i<1$, $\forall i$), to assure that all queues can handle the offered work.

\begin{remark}
Like in the two-dimensional case (cf. Remark \ref{rmk 2dedicated}), this model allows for a separate Poisson arrival stream into queue 1.
Merging this separate arrival process with the simultaneous arrival process,
the distribution of $(B^{(2)},\dots,B^{(K)})$ will have an atom in (0,\dots,0), which is the probability that a
dedicated Poisson arrival happens instead of a simultaneous one.
\newline Similarly, the model allows for simultaneous arrivals at the first $j$ queues only. This can be achieved by letting the distribution of $(B^{(j+1)},...,B^{(K)})$ have an atom at $(0,...,0)$.
\end{remark}

\paragraph{\textbf{The Laplace-Stieltjes transform of $(V^{(1)},\dots,V^{(K)})$}.}
Denote the Laplace-Stieltjes transform of the service time/claim size vector by
\[\phi(s_1,\dots,s_K) := \mean(e^{-s_1B^{(1)}-\dots-s_KB^{(K)}}).\]
We have the $K$-dimensional Lindley recursion for the random variables $(V^{(1)}_n,\dots,V^{(K)}_n):$
\begin{align*}
(V^{(1)}_{n+1},\dots,V^{(K)}_{n+1}) &=
({\rm max}(V^{(1)}_n + B^{(1)}_n - A_n,0),\dots,{\rm max}(V^{(K)}_n + B^{(K)}_n - A_n,0) )\;\;n\geq 1.
\end{align*}

For the LST:
\[
\psi_n(s_1,\dots,s_K) = \mean\left(e^{-s_1V^{(1)}_n-\dots-s_KV^{(K)}_n}\right), \quad n\geq 1,
\]
the Lindley recursion gives after straightforward calculations:
\begin{align}
\label{psi_nK}
\psi_{n+1}(s_1,\dots,s_K) &=  \sum_{j=1}^{K} \textstyle{\frac{\lambda}{\lambda-\sum_{i=1}^{j}s_i}\left[\phi^{(j)}(s_1,\dots,s_j)
\psi_n^{(j)}(s_1,\dots,s_j)\right.}\nonumber\\
&-
\left.\phi^{(j-1)}(s_1,\dots,s_{j-1})\psi_n^{(j-1)}(s_1,\dots,s_{j-1})\right] + \phi^{(0)}\psi_n^{(0)}
\end{align}

\noindent where we used the following notation for simplicity: $\psi_n^{(K)}(s_1,\dots,s_K):=\psi_n(s_1,\dots,s_K)$ and $\psi_n^{(0)}:=\psi_n(\lambda,0,\dots,0)$, and

\begin{align}\psi_n^{(j)}(s_1,\dots,s_j)&:=\psi_n(s_1,\dots,s_j,\lambda-\sum_{i=1}^js_i,\underbrace{0,\dots,0}_{K-j-1\;arguments}),\mbox{ for }1\leq j\leq K-1.\nonumber
\end{align}
$\phi^{(j)}(s_1,\dots,s_j)$ is analogously defined for $j=0,\dots,K$.
By taking $n\rightarrow\infty$ in (\ref{psi_nK}), we obtain for
$\psi(s_1,\dots,s_K) := \lim_{n \to \infty} \psi_n(s_1,\dots,s_K)$,
\begin{align}
\label{psiK}
\left(1-{\textstyle \frac{\lambda \phi(s_1,\dots,s_K)}{\lambda-\sum_{i=1}^Ks_i}}\right)\psi(s_1,\dots,s_K) &=
\sum_{j=0}^{K-1}\left({\textstyle \frac{\lambda}{\lambda-\sum_{i=1}^js_i}-\frac{\lambda}{\lambda-\sum_{i=1}^{j+1}s_i}}\right)\nonumber\\
&\cdot\phi^{(j)}(s_1,\dots,s_j)\psi^{(j)}(s_1,\dots,s_j),
\end{align}
with $\psi^{(j)}:=\displaystyle\lim_{n \rightarrow \infty}\psi_n^{(j)}$; and $\phi^{(0)}\psi^{(0)}=\mathbb P(V^{(1)}+B^{(1)}\leq A)=1- \rho_1 $.

Formula (\ref{psiK}) has a simple recursive structure, and we can rewrite it as:

\begin{align}  \left(1-{\textstyle \frac{\lambda \phi(s_1,\dots,s_K)}{\lambda-\sum_{i=1}^Ks_i}}\right)\psi(s_1,\dots,s_K)=\left({\textstyle \frac{\lambda}{\lambda-\sum_{i=1}^{K-1}s_i}-
\frac{\lambda}{\lambda-\sum_{i=1}^{K}s_i}}\right)\phi^{(K-1)}(s_1,\dots,s_{K-1})\cdot \nonumber \\ \psi^{(K-1)}(s_1,\dots,s_{K-1})
+ \left(1-{\textstyle \frac{\lambda \phi(s_1,...,s_{K-1},0)}{\lambda-\sum_{i=1}^{K-1}s_i}}\right)\psi(s_1,\dots,s_{K-1},0). \label{psiKrec}\end{align}

Denote by $C_j:=\left(1-{\textstyle \frac{\lambda \phi(s_1,\dots,s_j,0,\dots,0)}{\lambda-\sum_{i=1}^{j}s_i}}\right)\psi(s_1,\dots,s_j,0,\dots,0)$, and remark that\newline $\psi(s_1,\dots,s_j,0,\dots,0)$
is the transform of the workload in the $j$-dimensional system obtained by ignoring the last $(K-j)$ queues, $j=1,\dots,K$.


\begin{proposition}\label{Prop work_K} The LST of the steady-state workload in the $K\geq 3$ systems is given by:
\begin{equation}
\psi(s_1,\dots,s_K)=\frac{(1-\rho_K)(S_K-s_K)}{\sum_{i=1}^Ks_i-\lambda(1-\phi(s_1,\dots,s_K))}\prod_{j=2}^{K-1}\frac{1-\rho_j}{1-\rho_{j+1}}
\frac{S_j-s_j}{S_{j+1}}
\cdot\frac{1-\rho_1}{1-\rho_2}\frac{s_1}{S_2},
\label{Work_K}
\end{equation}
with $S_j=S_j(s_1,...,s_{j-1})$ the unique solution of the equation \[\lambda\phi(s_1,\dots,s_j,0,\dots,0)=\lambda-\sum_{i=1}^{j}s_i,\] with $\mathcal Re$ $(s_1+\dots+s_{j-1}+S_j(s_1,\dots,s_{j-1}))>0$,
for all $j=2,\dots,K$.

\end{proposition}

\begin{proof}
 The key remark is that $s_K$ is not among the arguments of the functions $\psi^{(j)}$ that appear in the righthand side of (\ref{psiK}).

From Lemma \ref{Rouche} applied to $s=s_1+\dots+s_{K-1}$ and $t=s_K$, there
exists a unique solution $S_K=S_K(s_1,\dots,s_{K-1})$ of the equation \[\lambda\phi(s_1,\dots,s_K)=\lambda-\displaystyle\sum_{i=1}^Ks_i,\]
such that $S_K(s_1,\dots,s_{K-1})+\sum_{i=1}^{K-1}s_i$ has positive real part.
Hence the hyper-surface given by $S_K=S_K(s_1,\dots,s_{K-1})$ is contained in the regularity domain of $\psi(s_1,\dots,s_K)$, and then
the righthand side of (\ref{psiKrec}) must be zero. This gives the following relation for $\psi^{(K-1)}(s_1,\dots,s_{K-1})$:   

\[ \phi^{(K-1)}(s_1,\dots,s_{K-1})\psi^{(K-1)}(s_1,\dots,s_{K-1})= \frac{(\lambda-\sum_{i=1}^{K-1}s_i)\phi(s_1,\dots,s_{K-1},S_K)}{ S_K}C_{K-1}. \]


By substituting back into Equation $(\ref{psiKrec})$, 
we obtain the recursion

\[ C_K=\frac{\lambda-\sum_{i=1}^{K-1}s_i}{\lambda-\sum_{i=1}^Ks_i}\cdot\frac{S_K-s_K}{S_K} C_{K-1}  \] with initial condition
$C_2=-(1-\rho_1)\frac{s_1}{\lambda-s_1-s_2}\frac{S_2-s_2}{S_2}$, which follows from $(\ref{psi_final})$. From this, the formula in $(\ref{Work_K})$ is obtained,
after rearranging the factors.\end{proof}


\vspace{0.2cm}

\paragraph{\textbf{Interpretation of the Rouch\'e root.}}

It is worthwhile to change the coordinates: $(s_1,s_2,\dots,s_K)\rightarrow (s_1,s_2,\dots,s_{K-1},\sum_{i=1}^Ks_i)$. We can rewrite

\[\phi(s_1,\dots,s_K)=\mathbb E e^{-s_1(B^{(1)}-B^{(K)})-...-s_{K-1}(B^{(K-1)}-B^{(K)})-\left(\sum_{i=1}^Ks_i\right)B^{(K)}}\]

Let us denote it by $\tilde\phi(s_1,\dots,s_{K-1},\sum_{i=1}^K s_i)$. This is the transform of the extra service time (relative to the shortest queue) in the first $K-1$ queues,
together with the shortest one.
It turns out there is a connection between $s_K(s_1,\dots,s_{K-1})$ and the joint extra work in systems 1 to $K-1$ at the end of a busy period in system $K$.
Let us denote this extra work by $(U_1,U_2,\dots,U_{K-1})$, with LST $U_{K}^*(s_1,\dots,s_{K-1})$, and let $F(x_1,x_2,\dots,x_{K})$ be the
multivariate c.d.f. of \newline $(B^{(1)}-B^{(K)},\dots,B^{(K-1)}-B^{(K)},B^{(K)})$. Then by a similar argument as the one leading to formula (\ref{U_star}), $U_{K}^*(s_1,\dots,s_{K-1})$ satisfies the identity

\begin{align}\label{branching_K}U_K^*(s_1,\dots,s_{K-1})&=\displaystyle{\int}e^{-\sum\limits_{i=1}^{K-1}s_ix_i}
\sum_{n=0}^\infty \frac{(\lambda x_K)^n}{n!}e^{-\lambda x_K}[U_K^*(s_1,\dots,s_{K-1})]^n F(dx_1\dots dx_K) \nonumber
\\ &=\tilde\phi(s_1,\dots,s_{K-1},\lambda[1-U_K^*(s_1,\dots,s_{K-1})]).\end{align}
Comparing this with the identity for the Rouch\'e root
 \[ \lambda -(s_1+\dots+s_{K-1}+S_K)=\lambda \tilde\phi(s_1,\dots,s_{K-1},s_1+\dots+s_{K-1}+S_K),\]
gives the relation analogous to (\ref{rel}) in  Proposition 1

\begin{equation}\label{Rouche_relK}\lambda U_K^*(s_1,\dots,s_{K-1})=\lambda-(s_1+\dots+s_{K-1}+S_K),\end{equation}
which follows because the Rouch\'e root is unique.



\vspace{0.2cm}

Let us fix our attention on the case $K=3$ for the moment. Then identity $(\ref{Work_K})$ becomes

\begin{align}\psi(s_1,s_2,s_3)&=\frac{(1-\rho_3)(S_3-s_3)}{s_1+s_2+s_3-\lambda[1-\phi(s_1,s_2,s_3)]}\cdot \frac{1-\rho_2}{1-\rho_3}\frac{S_2-s_2}{S_3}
 \cdot\frac{1-\rho_1}{1-\rho_2}\frac{s_1}{S_2}. \label{preservation work}\end{align}

\begin{figure}[!htb]
\centering
\includegraphics[scale=0.7,trim=3cm 10cm 3cm 2cm]{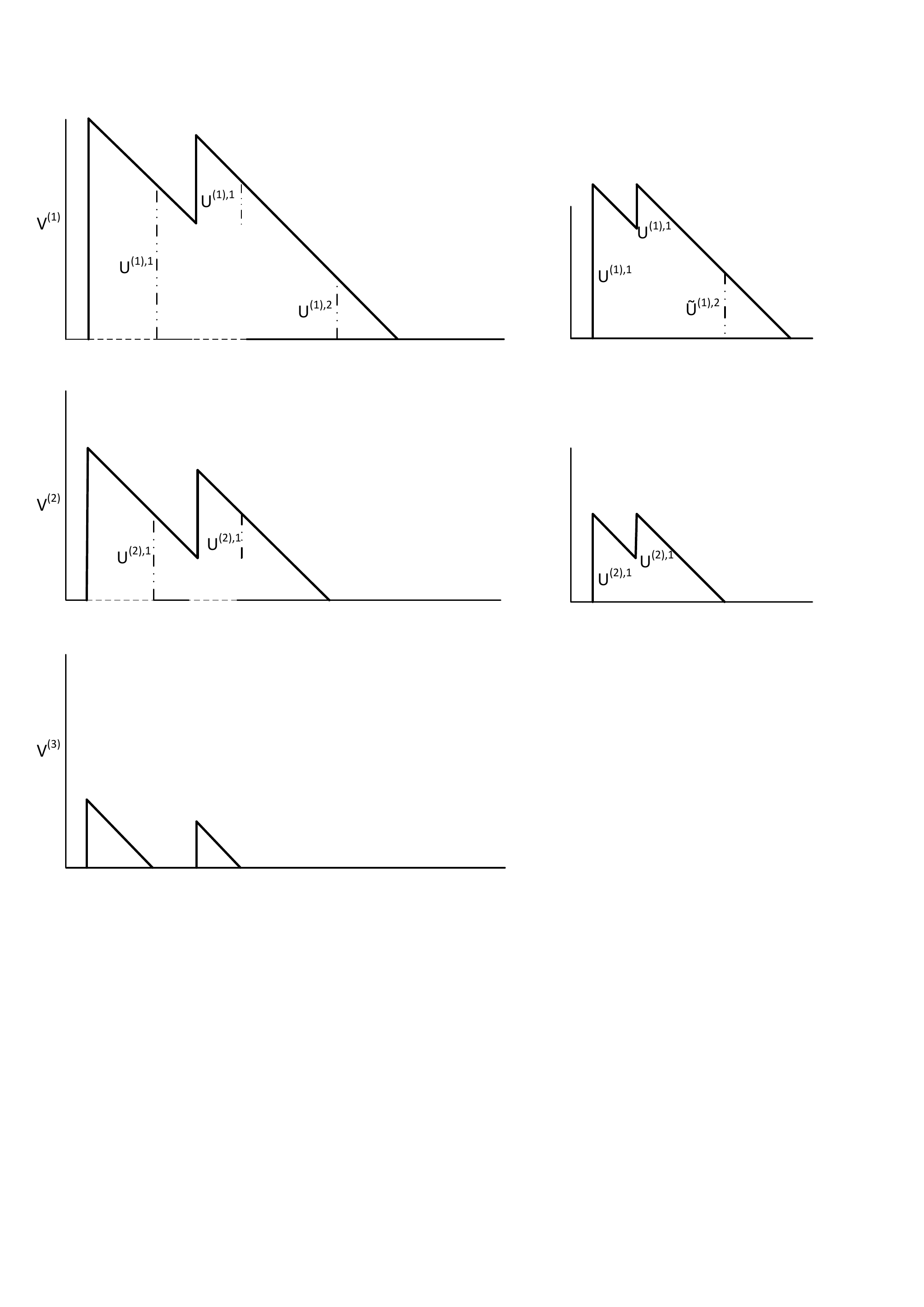}
\caption{Work in the original system (left) and in the virtual system (right)}
\label{Fig work pres}
\end{figure}

\begin{paragraph}{\textbf{Work conservation.}}
We would like to give a probabilistic interpretation of (\ref{preservation work}). In order to achieve this, we start by considering the joint extra work in queues 1 and 2 at the end of a busy period
in queue 3. This has LST $U_3^*(s_1,s_2)$ as input in a 2-dimensional system with simultaneous Poisson arrivals, which is obtained by contracting the busy cycles in queue 3. We call this the 2-dimensional virtual system. Remark that the inter-arrival times in the virtual system are precisely the idle periods in queue 3.

For this construction, the key observation is that the steady-state extra work in the virtual queue 1 at the end of the busy period in the virtual queue 2
is the same as the extra work in the initial queue 1 at the end of the busy period in the original queue 2.
In analytic form, let $\tilde U_2^*(s_1)$ be the LST of the extra work in the virtual system and $ U^*_2(s_1)$ be the LST of the extra work in the original system, see Figure \ref{Fig work pres}.

\begin{lemma}\label{Work conservation}

\[   \tilde U_2^*(s_1) =  U^*_2(s_1). \]

\end{lemma}
\begin{proof} We begin by remarking that the extra work $(U^{(1),1},U^{(2),1})$ in the first 2 queues at the end of a busy period
in queue 3 satisfies the a.s. inequality $U^{(1),1}\geq U^{(2),1}$. Since this is the input in the virtual system, from Proposition \ref{Prop_Rouche}, $\tilde U_2^*(s_1)$ satisfies the identity (\ref{U_star})
with $U_3^*(s_1,s_2)$ instead of $\phi(s_1,s_2)$:
 \begin{equation}\label{U_3^*}U^*_3(s_1,\lambda[1-\tilde U_2^*(s_1)]-s_1)=\tilde U_2^*(s_1).\end{equation}
\noindent At the same time, via (\ref{branching_K}),  $U_3^*(s_1,s_2)$ satisfies

\[\phi(s_1,s_2,\lambda(1-U_3^*(s_1,s_2))-s_1-s_2)=U_3^*(s_1,s_2).\]
If we substitute this fixed point identity in (\ref{U_3^*}) above, we have

\[ \phi(s_1,\lambda(1-\tilde U_2^*(s_1))-s_1,0)=\tilde U_2^*(s_1). \]

\noindent On the other hand, this is also the identity (\ref{U_star}) satisfied by $U_2^*(s_1)$, in the 2-dimensional system obtained by ignoring the last queue.
Hence, from the uniqueness result in Lemma \ref{Rouche}, $\tilde U_2^*(s_1)=U_2^*(s_1)$ (See Figure \ref{Fig work pres}). This completes the proof.\end{proof}

\noindent We can rewrite $(19)$ using $(21)$:

\begin{align} \label{psi_final3} \psi(s_1,s_2,s_3)&=(1-\rho_3)\frac{s_1+s_2+s_3-\lambda(1-U^*_3(s_1,s_2))}{s_1+s_2+s_3-\lambda(1-\phi(s_1,s_2,s_3))}
\nonumber\\ &\cdot \frac{1-\rho_2}{1-\rho_3}\frac{s_1+s_2-\lambda(1-U_2^*(s_1))}{s_1+s_2-\lambda(1-U_3^*(s_1,s_2))}\cdot \frac{1-\rho_1}{1-\rho_2}\frac{s_1}{s_1-\lambda(1-\tilde U_2^*(s_1))}.\end{align}

Remark that the atom $\frac{1-\rho_1}{1-\rho_2}$ above is the conditional probability that queue 1 is empty, given that queue 2 is empty; and similarly for $\frac{1-\rho_2}{1-\rho_3}$.
In addition, the last factor in (\ref{psi_final3}) is the Pollaczek-Khinchine representation for an M/G/1 queue with service times having LST $\tilde U_2^*(s_1)$.
Now we are ready to give the main result of this section.
\end{paragraph}
\begin{theorem} In steady state, the joint workload distribution decomposes as an independent sum:

\begin{align}(V^{(1)},V^{(2)},V^{(3)})&\deq(\tilde V^{(1),1},\tilde V^{(2),1},V^{(3)})+ (\tilde V^{(1),2}, V^{(2),2},0) 
+ (V^{(1),3},0,0).\nonumber\end{align}
The first term in the sum represents the steady-state distribution of the modified joint workload process obtained by removing the extra work in the first two queues at the end of
a busy period in the third queue. The second term is the workload in the first two queues obtained by removing the extra work in the first queue at the end of a
busy cycle in the second queue. Finally the third term represents the workload in the virtual M/G/1 queue with input distributed as the extra work in queue 1, at the end of a busy period in queue 2.

\end{theorem}

\begin{proof}
Consider the modified work process that evolves in steady state as

\[ (\tilde V^{(1),1},\tilde V^{(2),1},V^{(3)})\deq \left(\tilde V^{(1),1}+B^{(1)}-A,\tilde V^{(2),1}+B^{(2)}-A,V^{(3)}+B^{(3)}-A\right),\]
if $A< V^{(3)}+B^{(3)}$; and $(\tilde V^{(1),1},\tilde V^{(2),1},V^{(3)})=(0,0,0)$, else.

By similar computations as the ones leading to Formula $(\ref{psi_tilde})$, we obtain

\[ \tilde\psi(s_1,s_2,s_3)=(1-\rho_3)\frac{s_1+s_2+s_3-\lambda(1-U^*_3(s_1,s_2))}{s_1+s_2+s_3-\lambda(1-\phi(s_1,s_2,s_3))}.
\]

This is the first factor in $(\ref{psi_final3})$. For the second one, consider the following modified virtual workload process that evolves in steady state as

\[ (\tilde V^{(1),2}, V^{(2),2},0)\deq\left\{\begin{array}{ll}\left(\tilde V^{(1),2}+U^{(1),1}-A, V^{(2),2}+U^{(2),1}-A,0\right), &\mbox{if}\; A< V^{(2),2}+U^{(2),1}, \\
(0,0,0), &\mbox{if}\; A\geq V^{(2),2}+U^{(2),1},\end{array}\right.  \]
with $(U^{(1),1},U^{(2),1})$ the extra work vector in the first 2
queues  at the end of a busy period in queue 3.
Here we remove the excess workload in the virtual queue 1 at the end of the busy period in the virtual queue 2, which by Lemma \ref{Work conservation} is the same as in the original system.
In terms of LST's , this becomes

\[\tilde \psi_1(s_1,s_2)=\frac{1-\rho_1}{1-\rho_2}\frac{s_1+s_2-\lambda(1-U_2^*(s_1))}{s_1+s_2-\lambda(1-U^*_3(s_1,s_2))}.\]

 \noindent Finally, the third factor in (\ref{psi_final3}) is the Pollaczek-Khinchine representation of the steady-state workload in the M/G/1 queue with service time distributed as the extra work in queue 1 at the end of a busy period in queue 2. This ends the proof. \end{proof}
 These considerations can be iterated now for the general $K$-dimensional system.

\begin{corollary}

 The steady-state joint workload in the K systems decomposes into the independent sum

\begin{align*}(V^{(1)},\dots,V^{(K)}) &\deq (\tilde V^{(1),1},\dots,\tilde V^{(K-1),1},V^{(K)})+ (\tilde V^{(1),2},\dots,\tilde V^{(K-2),2}, V^{(K-1),2},0) \nonumber \\
 &+\dots+(\tilde V^{(1),K-1}, V^{(2),K-1},0,\dots,0) + (V^{(1),K},0,\dots,0),\nonumber\end{align*}
where the $j$th term  in the sum satisfies the identity in distribution (j=2,\dots,K):

\begin{align*} (\tilde V^{(1),j},&\tilde V^{(2),j},\dots,\tilde V^{(K-j),j}, V^{(K-j+1),j},0,\dots,0) \deq \left(\tilde V^{(1),j}+U^{(1),j-1}-A,\right.\nonumber\\
&\tilde V^{(2),j}+U^{(2),j-1}-A,\dots\left., V^{(K-j+1),j}-B^{(K-j+1)}-A,0,\dots,0 \right), \nonumber \\
&\mbox{if }\;A\leq V^{(K-j+1),j}-B^{(K-j+1)}, \end{align*}

and $(0,\dots,0)$ else. 
$U^{(i),j}$ is the extra workload in queue $i$ at the end of a busy period in queue $(K-j+1)$, for $i>K-j+1$.
\end{corollary}

\end{section}

\begin{section}{The general two-dimensional workload/reinsurance problem\label{Section boundary value}}

In this section we consider the general two-dimensional workload problem:
pairs of customers arrive simultaneously at two parallel queues $Q_1$ and $Q_2$ according to a Poisson($\lambda$) process,
the $n$th pair requiring service times $(B_n^{(1)},B_n^{(2)})$ with LST $\phi(s,t)$.
We are interested in the steady-state workload vector $(V^{(1)},V^{(2)})$ with LST $\psi(s,t)$.
By the duality that is exposed in Section~\ref{Section duality}, $\psi(s,t)$ also is the
Laplace transform (w.r.t.\ $u_1$ and $u_2$) of the probability
that
both portfolios of an insurance company with simultaneous claims
$(B_n^{(1)},B_n^{(2)})$,
with initial capital $u_1$ and $u_2$, will survive.

In Section~\ref{SectionK_2} we have determined $\psi(s,t)$ for the special case that
$\mathbb P(B^{(1)} \geq B^{(2)}) = 1$.
We now show how the general case -- $B_n^{(1)}$ and $B_n^{(2)}$ having an arbitrary
joint distribution --
has been solved in the literature (with the solution of that special case emerging as a degenerate
solution). We shall successively discuss the contributions of Baccelli \cite{Baccelli},
De Klein \cite{DeKlein} and Cohen \cite{Cohen92},
who have treated the two-dimensional workload problem  with simultaneous arrivals in increasing generality.
Starting point in all those three studies is the following functional equation for $\psi(s,t)$, which is derived by
studying the $2$-dimensional Markovian workload process during an infinitesimal amount of time $\Delta t$:
\begin{equation}
K(s,t) \psi(s,t) = t \psi_1(s) + s \psi_2(t), ~~~~~~~~~\mathcal Re ~ s,t \geq 0.
\label{funda}
\end{equation}
Here the so-called {\em kernel} $K(s,t)$ is given by:
\begin{equation}
K(s,t) := s + t - \lambda (1 - \phi(s,t)),
\label{kernel}
\end{equation}
and
\begin{equation}
\psi_1(s) := \mathbb E[{\rm e}^{-s V_1} (V_2=0)], ~~~\psi_2(t) := \mathbb E[{\rm e}^{-t V_2}(V_1=0)],
\end{equation}
with $(\cdot)$ denoting an indicator function.
\\

\noindent
\begin{remark}
In the special case of Section~\ref{SectionK_2}, with $\mathbb P(B^{(1)} \geq B^{(2)}) =1$,
one has $\psi_2(t) \equiv \mathbb P(V_1=0)$, because $V_2$ cannot be positive when $V_1=0$.
It then remains to find $\psi_1(s)$.
This is done by observing (cf.\ the appendix) that, for all $s$ with $\mathcal Re ~s > 0$, there is a unique zero $t(s)$
of the kernel, with $\mathcal Re ~t(s) > \mathcal Re ~(-s)$.
This immediately yields that $\psi_1(s) =  - \frac{s}{t(s)} \mathbb P(V_1=0)$, which is readily seen to be
in agreement with (\ref{psi_final}).

Equation (\ref{psi}), which was obtained by studying the workloads at arrival epochs
(i.e., the waiting times; by PASTA they have the same distribution as the steady-state workloads),
looks slightly different from (\ref{funda}), but using (\ref{intermed_rel}) it is readily seen that they are equivalent.
\end{remark}
\noindent
Globally speaking, the  essential steps in \cite{Baccelli,DeKlein,Cohen92} are the following.
\\
{\em Step 1}: find a suitable set of zeroes $(\hs,\hatt)$, with $\mathcal Re ~ \hs \geq 0$, $\mathcal Re ~\hatt \geq 0$,
of the kernel $K(s,t)$, i.e.,
$K(\hs,\hatt)=0$.
Because $\psi(s,t)$ is regular for all $(s,t)$ with
$\mathcal Re ~ s,t \geq 0$,
one must have  for all these zeroes:
\begin{equation}
\hatt \psi_1(\hs) = -  \hs \psi_2(\hatt).
\label{funda2}
\end{equation}
It is further observed that $\psi_1(s)$ is regular for $\mathcal Re ~s > 0$, continuous for $\mathcal Re  ~s \geq 0$,
and that
$\psi_2(t)$ is regular for $\mathcal Re ~t > 0$, continuous for $\mathcal Re ~t \geq 0$.
\\
{\em Step 2}: formulate a boundary value problem for $\psi_1(s)$ and $\psi_2(t)$.
There are various types of boundary value problems, like the Riemann and the Wiener-Hopf
boundary value problems. Typically, they ask to determine two functions $P_1(\cdot)$
and $P_2(\cdot)$, which satisfy a relation on a particular boundary $B$,
while $P_1(\cdot)$ is regular in the interior $B^+$ and $P_2(\cdot)$ is regular in the exterior $B^-$.
$B$ could be the unit circle (Riemann boundary value problem), or the imaginary axis
(Wiener-Hopf boundary value problem; $B^+$ now is the left-half plane).
We refer to Gakhov \cite{Gakhov} and Mushkelishvili \cite{Mushkelishvili} for excellent expositions
of such boundary value problems and their variants, like the boundary value
problem with a shift. The latter occurs in the approach of De Klein \cite{DeKlein}, see below.
\\
{\em Step 3}: solve the boundary value problem for $\psi_1(\cdot)$ and $\psi_2(\cdot)$ with boundary $B$.
If $B$ is a smooth closed contour that is not a circle,
the use of a conformal mapping from $B$ to the unit circle $C$ is required to arrive at a Riemann
boundary value problem for the unit circle, the solution of which can be found in \cite{Gakhov,Mushkelishvili}.
Thus one obtains $\psi_1(s)$ and $\psi_2(t)$ inside certain regions; subsequently, one may use analytic continuation to find
them in $\mathcal Re ~ s,t \geq 0$.
Finally, $\psi(s,t)$ follows from (\ref{funda}).
\begin{remark}
Application of the boundary value method in queueing theory
was pioneered by Fayolle and Iasnogorodski in \cite{FI}.
They used this method to analyze the joint queue length process in two coupled processors, viz., two $M/M/1$ queues
which operate at unit speeds when the other queue is not empty, but at different speeds when the other queue
is empty. The method was subsequently developed in \cite{BVP} for a large class of two-dimensional random walks;
various queueing applications were also discussed in \cite{BVP}.
See \cite{Cohen88} for a survey of the method in queueing theory, and see \cite{FM,Cohen92}
for two monographs which have further developed the theory of two-dimensional random walks.
Part IV of \cite{Cohen92} explores the analysis
of $N$-dimensional random walks with $N>2$. Results for $N>2$ are very limited,
and it seems fair to conclude that the boundary value method
is, apart from a few special cases, restricted to two-dimensional random walks.
\end{remark}
\begin{remark}
We strongly believe that the boundary value method also has a large potential in
the analysis of two-dimensional risk models.
Due to the duality between the reinsurance model and the $2$-queue model
with simultaneous arrivals,
the publications \cite{Baccelli,DeKlein,Cohen92}
are of immediate relevance to the reinsurance problem.
These publications seem unknown in the insurance community (see, e.g., Chan et al.\ \cite{Chan},
who pose the two-dimensional risk problem and stop at Equation  (\ref{funda}) (where \cite{Baccelli,DeKlein,Cohen92} begin).
They have remained largely unnoticed even in the queueing community,
perhaps because of their complexity and because \cite{Baccelli} and \cite{DeKlein} did not appear
in the open literature.
For these reasons,
we now successively expose the approaches in \cite{Baccelli}, \cite{DeKlein} and \cite{Cohen92} at some length.
\end{remark}
{\em The approach of Baccelli \cite{Baccelli}}
\\
Baccelli \cite{Baccelli} restricts himself to the case of exchangeable $(B^{(1)},B^{(2)})$, i.e.,
$\mathbb P(B^{(1)}<x,B^{(2)} <y)$ $=
\mathbb P(B^{(1)}<y,B^{(2)} <x)$, or equivalently, $\phi(s,t) = \phi(t,s)$.
We briefly review the three steps mentioned above.
\\
{\em Step 1} in \cite{Baccelli} is as follows. Consider zero pairs $(\hs,\hatt) = (g+iu,g-iu)$ of kernel $K(s,t)$, with $u \in \mathcal R$
and with $g = g(u)$ the unique zero in $\mathcal Re ~g \geq 0$ of
\[
2g = \lambda(1 - \phi(g+iu,g-iu)) .
\]
Using the exchangeability, it can be shown that this unique zero is real and non-negative, while $g(-u) = g(u), ~ u \in \mathcal R$.
\\
{\em Step 2}.
Consider the arc $A = \{s: s=g(u) + iu, ~ u \in R\}$, with $g(u)$ the zero defined above.
This is a smooth arc, located in the right half-plane.
Baccelli finds a conformal mapping $p(\cdot)$ of
the interior $C^+$ of the unit circle $C$ onto
$A^+$, the `interior' of $A$ located on the right of $A$,
and a conformal mapping
$q(\cdot)$ of $C^-$, the exterior of the unit circle,
onto $A^+$; their limits on $C$ are denoted by $p^+(z)$ and $q^-(z)$, which are each other's complex conjugates
because of the exchangeability.
Noticing that $p^+(-1)=q^-(-1)=0$, he
multiplies both sides of (\ref{funda2}) with $1+z$. This yields (divide both sides of (\ref{funda2}) by $\hs \hatt$):
\begin{equation}
(1+z) \frac{\psi_1(p^+(z))}{p^+(z)} = - (1+z) \frac{\psi_2(q^-(z))}{q^-(z)} , ~~~~~~~~|z| = 1.
\label{funda3}
\end{equation}
Because of the regularity properties of the conformal mappings and of $\psi_1(s)$ and $\psi_2(t)$,
$\mathcal Re ~s,t > 0$, one now arrives at a simple boundary value problem:
we have (\ref{funda3})  for $|z|=1$, while
the left-hand side of (\ref{funda3}) is regular for $|z| <1$, and the right-hand side is regular for
$|z| >1$.
\\
{\em Step 3}.
The solution of this problem immediately follows from Liouville's theorem, cf.\ \cite{Titchmarsh} p. 85:
\[
\psi_1(p(z)) = \frac{\gamma + \delta z}{1+z} p(z), ~~ |z| <1, ~~~
\psi_2(q(z)) = \frac{\gamma + \delta z}{1+z} q(z), ~~~~~~ |z| >1.
\]
Baccelli \cite{Baccelli} shows that $\gamma = - \delta$, and determines the remaining unknown constant
$\delta$ by normalization.
Having thus determined $\psi_1(s)$ for $s \in A^+$, he uses analytic continuation to obtain $\psi_1(s)$
in the whole right half-plane; similarly for $\psi_2(t)$.
Finally, substitution in (\ref{funda}) determines $\psi(s,t)$.
\\

\noindent
{\em The approach of De Klein \cite{DeKlein}}
\\
De Klein \cite{DeKlein}, pp. 119-168, studies the general case of an arbitrary joint distribution of $B^{(1)}$ and $B^{(2)}$.
\\
{\em Step 1} in \cite{DeKlein} is as follows.
He considers the same zero pairs as Baccelli (also suggesting another set of zero pairs on p. 132).
$g(u)$ is no longer necessarily real, but for all real $u$ there still is a unique zero $g(u)$.
\\
{\em Step 2}.
De Klein subsequently considers the simple, smooth arcs $A_1 = \{s: s = g(u) + iu, ~ u \in R\}$ and
$A_2 = \{t: t = g(u) - iu, ~ u \in R\}$ in the right half-plane.
Notice that $A_1$ and $A_2$ are each other's complex conjugates in the exchangeable case of Baccelli, but not in De Klein's more general case.
De Klein now uses the (unique) one-to-one mapping $t = \omega_2(s)$ from $A_1$ onto $A_2$ (with inverse
$\omega_1(t)$) determined by the fact that, $\forall s \in A_1$, $(s,\omega_2(s))$ is a zero pair of the kernel.
Similarly, $\forall  t \in A_2$, $(\omega_1(t),t)$ is a zero pair.
Hence the following must hold:

\begin{equation}
\psi_1(\omega_1(t)) = - \frac{\omega_1(t)}{t} \psi_2(t), ~~~t \in A_2.
\label{funda4}
\end{equation}
In addition, one has the regularity properties of the functions $\psi_1(\cdot)$ and $\psi_2(\cdot)$
which were listed below (\ref{funda2}).
Determination of functions $\psi_1(\cdot)$ and $\psi_2(\cdot)$
with these regularity properties and satisfying (\ref{funda4})
is a so-called {\em shift problem}, a boundary value problem with a shift (cf.\ Sections 17 and 18  of \cite{Gakhov}).
\\
{\em Step 3}.
Gakhov \cite{Gakhov} mentions two methods to solve such problems: (i) reduce the problem to a Fredholm integral equation
of the second kind, and (ii) reduce the problem to an ordinary Riemann boundary value problem, by means
of conformal mappings.
De Klein \cite{DeKlein} explores the first method in Section II.4.2 and the second in Section II.4.3.
We concentrate on the first method.
De Klein first translates the shift problem to one on a finite smooth {\em closed} contour,
via the conformal mapping $\zeta(z) = \frac{1-z}{1+z}$ (with inverse $z(\zeta) = \frac{1-\zeta}{1+\zeta}$) that maps $A_i$ onto smooth closed contours $T_i$, $i=1,2$;
he then applies Gakhov's first method. He obtains the following Fredholm integral equation of the second kind for an unknown function $G_1(\cdot)$ -- which up to a constant
equals ${\rm log} \{\psi_1(z(\cdot))\}$:
\begin{equation}
G_1(p_1) = \frac{1}{2 \pi i} \int_{T_1} G_1(v_1)
[\frac{1}{v_1 - p_1} - \frac{\nu_2'(v_1)}{\nu_2(v_1) - \nu_2(p_1)} - \frac{1}{v_1 - c_1}] {\rm d} v_1
+ H_1(p_1), ~~~ p_1 \in T_1,
\end{equation}
with $H_1(\cdot)$ some known function,
$c_1$ some point in the interior of $T_1$,
and $\nu_2(v_1)$ $= \zeta(\omega_2(z(v_1)))$, $v_1 \in T_1$.
After having solved the integral equation (which can be done numerically in an efficient way,
as shown by De Klein), one obtains $\psi_1(s)$ for $s \in T_1$, and then $\psi_2(t)$ for $t \in T_2$ via (\ref{funda2}).
The regularity of $\psi_1(s)$ in the interior $T_1^+$ subsequently allows one to
obtain $\psi_1(s)$, $s \in T_1^+$, as a Cauchy integral; similarly for $\psi_2(t)$, $t \in T_2^+$.
By analytic continuation, $\psi_1(s)$ and $\psi_2(t)$ are then also uniquely determined in $\mathcal Re ~s \geq 0$
and $\mathcal Re ~t \geq 0$, respectively. Finally, $\psi(s,t)$ again follows from (\ref{funda}).

De Klein also explores Gakhov's second method to treat the shift problem.
However, this reduction to a Riemann boundary value problem requires a conformal mapping
that itself must be determined by solving another Fredholm integral equation of the second kind.
In Chapter II.6 he extensively investigates the numerical solution of both integral equations
by means of the Nystrom or quadrature method.
He obtains, a.o., accurate results for the mean sojourn time of a customer pair, viz., the time
until both customers of a pair have left the system.
\\

\noindent
{\em The approach of Cohen \cite{Cohen92}}
\\
Cohen \cite{Cohen92}, Part III, considers
a very general class of two-dimensional workload processes. Basically, he combines the model with simultaneous arrivals
and the coupled processors model. The two servers have speeds $r_1$ and $r_2$ if they are both non-idle,
and speeds $r^{(1)}$ and $r^{(2)}$ when the other server is idle.
Furthermore, he also allows the possibility of different joint service requirement distributions
if a customer pair arrives when at least one of the servers is idle.
Finally, he explicitly allows single arrivals next to simultaneous arrivals
(cf.\ also \cite{Badescu}).
Much of Part III of \cite{Cohen92} is devoted to a detailed study
of the ergodicity conditions and of the so-called hitting point process and hitting point identity of the workload process,
hitting point referring to the first entrance point of one of the axes.

In Chapter III.4 he determines the steady-state joint workload distribution for a variety of cases.
For us, the most relevant cases are treated in Sections III.4.9 and III.4.10.
Section III.4.9 treats the model of De Klein \cite{DeKlein}.
The same zero pairs are used ({\em Step 1}), and the same smooth closed contours $T_1$ and $T_2$;
Cohen subsequently uses Gakhov's second method to arrive at a Riemann boundary value problem ({\em Step 2}).
That boundary value problem actually is so simple that it can be solved straightforwardly by applying Liouville's theorem, cf. Baccelli's method above ({\em Step 3});
however, a conformal mapping is required, which is obtained as the solution of another Fredholm integral equation of the second kind.
A nice feature in Section III.4.9 is that $\psi_1(s)$ and $\psi_2(t)$, after normalization, are
expressed as LST's of waiting time or workload distributions of special $M/G/1$ queues
(which are related to hitting points).

Section III.4.10 treats the model of De Klein with the additional feature that there is coupling of the servers, of a rather special form:
$\frac{r_1}{r^{(1)}} + \frac{r_2}{r^{(2)}} = 1$.
This does not change the kernel $K(s,t)$ (which only refers to the interior of the state space, with both servers active),
so the same zero pairs and contours can still be used.
However, it does change the right-hand side of (\ref{funda}), and hence
a slightly different Riemann boundary value problem must be solved.
\begin{remark}
It should be observed that Baccelli \cite{Baccelli}, De Klein \cite{DeKlein} and Cohen \cite{Cohen92} all also
solve the more complicated {\em transient} problem, of determining the joint time-dependent distribution
of the two workloads.
\end{remark}
\end{section}

\begin{section}{Conclusions and future work}\label{section comments}

We have studied a multivariate queueing system, which is shown to correspond to a dual risk
process with multiple lines of insurance that receive coupled claims. We find the
LST of the multivariate workload distribution in the case in which the service requirements
are ordered with probability one. Duality then yields the Laplace transform
of the survival probabilities.
For general service requirement (resp. claim size) vectors the workload (resp. ruin) problem can be solved
in the two-dimensional case, by solving a Riemann boundary value problem.
For dimension $K>2$, the problem seems analytically intractable in its full generality.
That raises the need for approximations and asymptotics.
It would in particular be interesting to obtain explicit multi-dimensional tail asymptotics of workloads
and ruin probabilities, both for light-tailed and heavy-tailed service requirements (or claim sizes).
Even for $K=2$ queues, this is already quite challenging. Moreover, a wide range of different cases
must be studied, giving rise to quite different techniques and results.
Therefore we intend to devote a separate study to tail asymptotics.

\end{section}

\begin{section}{Appendix\label{Appendix}}

\begin{lem}[Rouch\'{e} zero]\label{Rouche}For every s with $\mathcal Re\;s>0$  there exists a unique $t=t(s)$ with $\mathcal Re\;t(s)>\mathcal Re\;(-s)$, that satisfies the identity
\[ \lambda \phi(s,t)=\lambda-(s+t). \] Moreover the function: $s\rightarrow t(s)$ is analytic in $\mathcal Re\;s>0$.\end{lem}

\begin{proof} For fixed $s$ with $\mathcal Re\; s>0$, let $f(s+t):=\lambda-(s+t)$. Consider in the right half-plane the contour $\mathcal C$ made up from the semicircle with center at $-s$
and radius $R>2\lambda$ together with the line segment $I:=\left\{-s+iw|w\in[-R,R]\right\}$. We show that on this contour $|\lambda\phi(s,t)|<|f(s+t)|$. We can bound $|\phi(s,t)|$ by

\[ \lambda \left|\phi(s,t)\right|=\lambda |\tilde\phi(s,s+t)| \leq \lambda  \mathbb E e^{- Re\;s(B^{(1)}-B^{(2)})- Re\;(t+s)B^{(2)}}<\lambda. \]
This holds everywhere in the domain of $\phi(s,t)$ if $B^{(1)}-B^{(2)}$ has positive mass on $(0,\infty)$.

Now we bound $|f(s+t)|$. When $(s+t)$ is on the semicircle (i.e $|s+t|=R>2\lambda$), apply the triangle inequality to the triangle with vertices at $0,\lambda,s+t$, to find $|\lambda-s-t|>\lambda$.
When $(s+t)\in I$, by a similar argument we obtain $|\lambda-s-t|\geq \lambda$, with equality only when $s+t=0$. Hence on the contour $\mathcal C$, $|f(s+t)|\geq \lambda$.
We can now use Rouch\'e's theorem to conclude that the equation $ \lambda \phi(s,t)=\lambda-(s+t)$ has a unique solution $t(s)$ inside $\mathcal C$,
because the polynomial $f(s+t)$ has only one zero inside $\mathcal C$, at $\lambda$. Letting $R\rightarrow \infty$, proves the assertion. \end{proof}

\vspace{0.2cm}

\end{section}

\end{document}